\documentclass[12pt]{article}
\usepackage{graphicx}
\usepackage{amsmath,amsthm,amssymb,amsfonts,euscript,enumerate}
\usepackage{amsthm}
\usepackage{color,wrapfig}
\usepackage{pxfonts}
\usepackage{verbatim}
\newtheorem{thm}{Theorem}[section]

\newtheorem{lem}[thm]{Lemma}
\newtheorem{prop}[thm]{Proposition}

\newtheorem{rmk}[thm]{Remark}

\newcommand{\R}{{\mathbb{R}}}

\newcommand{\1}{\partial}

\newcommand{\4}{\widetilde}

\begin{document}
\title{Another proof of the local curvature estimate for the Ricci flow}
\author{Shu-Yu Hsu\\
Department of Mathematics\\
National Chung Cheng University\\
168 University Road, Min-Hsiung\\
Chia-Yi 621, Taiwan, R.O.C.\\
e-mail: shuyu.sy@gmail.com}
\date{Jan 18, 2018}
\smallbreak \maketitle
\begin{abstract}
By using the De Giorgi iteration method we will give a new simple proof of the recent result of B.~Kotschwar, O.~Munteanu, J.~Wang \cite{KMW} and N.~Sesum \cite{S} on the local boundedness of the Riemannian curvature tensor of solutions of Ricci flow in terms of its inital value on a given ball and a local uniform bound on the Ricci curvature.
\end{abstract}

\vskip 0.2truein

Key words: Ricci flow, local boundedness, Riemmanian curvature, Ricci curvature

AMS 2010 Mathematics Subject Classification: Primary 58J35, 35B45
Secondary 35K10 

\vskip 0.2truein
\setcounter{section}{0}

\section{Introduction}
\setcounter{equation}{0}
\setcounter{thm}{0}

There is a lot of interest on Ricci flow (\cite{CK}, \cite{CLN}, \cite{H2}, \cite{MF}, \cite{MT}) because it is a very powerful tool in the study of the geometry of manifolds. Recently G. Perelman \cite{P1}, \cite{P2}, by using the 
the Ricci flow technique solved the famous Poincare conjecture in geometry. Let $(M,g(t))$, $0<t<T$, be a $n$-dimensional Riemannian manifold. We say that the metric $g(t)=(g_{ij}(t))$ evolves by the Ricci flow if it satisfies
\begin{equation}\label{rf-eqn}
\frac{\1 g_{ij}}{\1 t}=-2R_{ij}
\end{equation} 
on $M\times (0,T)$ where $R_{ij}$ is the Ricci curvature of the metric $g(t)=(g_{ij}(t))$. Short time existence of solution of Ricci flow on compact Riemannian manifolds with any initial metric at $t=0$ was proved by R.~Hamilton in \cite{H1}. Short time existence of solution of Ricci flow on complete non-compact manifolds with bounded curvature initial metric at time $t=0$ was proved by W.X.~Shi in \cite{Sh1}, \cite{Sh2}. When $M$ is a compact manifold, R.~Hamilton \cite{H1} proved that either the Ricci flow solution exists globally or there exists a maximal existence time $0<T<\infty$ for the solution of Ricci flow and 
\begin{equation*}
\lim_{t\nearrow T}|Rm|_{g(t)}=\infty.
\end{equation*}
Hence in order to know whether the solution of Ricci flow can be extended beyond its interval of existence $(0,T)$, it is important to prove boundedness of the Riemannian curvature for the solution of Ricci flow near the time $T$. Uniform boundedness of the Riemannian curvature of the solution of Ricci flow on a compact manifold when the solution has uniform bounded Ricci curvature on $(0,T)$ was proved by N.~Sesum in \cite{S} using a blow-up contradiction argument and G.~Perelman's  noncollapsing result \cite{P1}. Local boundedness of the Riemannian curvature for $\kappa$-noncollapsing solutions of Ricci flow  in term of its local $L^{\frac{n}{2}}$ norm when its local $L^{\frac{n}{2}}$ norm is sufficiently small was also proved by R.~Ye in \cite{Y1}, \cite{Y2}, using Moser iteration technique and the point picking technique of G.~Perelman \cite{P1}. Similar result was also obtained by X~Dai, G~Wei and R~Ye in \cite{DWY}.

Local boundedness of the Riemannian curvature of the solution of Ricci flow in terms of its inital value on a given ball and a local uniform bound on the Ricci curvature was proved by B.~Kotschwar, O.~Munteanu, J.~Wang using Moser iteration technique and results of P.~Li \cite{L} in \cite{KMW}. A similar local Riemannian curvature result was proved recently by C.W.~Chen \cite{C} using the point picking technique of Perelman \cite{P1}, M.T.~Anderson's harmonic coordinates \cite{A} and elliptic regularity results \cite{GT}. In this paper we will use the De Giorgi iteration method to give a new simple proof of this result.

We will assume that $(M,g(t))$ is a smooth solution of the Ricci flow \eqref{rf-eqn} in $[0,T)$ for the rest of the paper. For any $x_0\in M$, $\rho>0$ and $0\le t<T$, we let $B_{g(t)}(x_0,\rho)=\{x\in M:\mbox{dist}_{g(t)}\, (x,x_0)<\rho\}$, $V_{x_0}(\rho,t)=\mbox{vol}_{g(t)}\,(B_{g(0)}(x_0,\rho))$, $V_{x_0}(\rho)=V_{x_0}(\rho,0)$, $|Ric|(x,t)=|Ric(x,t)|_{g(t)}$ and $|Rm|(x,t)=|Rm(x,t)|_{g(t)}$. We let $dv_t$ be the volume element of the metric $g(t)$ and let $C>0$ denote a generic constant that may change from line to line. For any complete Riemannian manifold $(M,g)$, we let $B(x_0,\rho)=\{x\in M:\mbox{dist}_g\, (x,x_0)<\rho\}$, $V_{x_0}(\rho)=\mbox{vol}_g\,(B(x_0,\rho))$ and $dv$ be the volume element of the metric $g$.

Note that by Corollary 13.3 of \cite{H1} or  Lemma 7.4 of \cite{CK},
\begin{equation}\label{rm-curv-eqn1}
\frac{\1}{\1 t}|Rm|^2\le\Delta |Rm|^2-2|\nabla Rm|^2+C|Rm|^3
\end{equation}
in $(0,T)$ for some constant $C>0$ depending only on $n$. Since $|\nabla|Rm||\le |\nabla Rm|$, by \eqref{rm-curv-eqn1},
\begin{equation}\label{rm-curv-eqn2}
\frac{\1}{\1 t}|Rm|\le\Delta |Rm|+C|Rm|^2\quad\mbox{ in }(0,T).
\end{equation}

We will prove the following main result in this paper.

\begin{thm}(cf. Theorem 1 of \cite{KMW})\label{lp-l-infty-thm}
Let $g(t)$, $0\le t<T$, be a smooth solution of Ricci flow on a $n$-dimensional Riemannian manifold $M$.  Suppose there exists $x_0\in M$ and constants $K>0$, $\rho>0$, such that
\begin{equation}\label{ric-curv-bd-cond}
|Ric|\le K\quad\mbox{ in }B_{g(0)}\left(x_0,\frac{2\rho}{\sqrt{K}}\right)\times [0,T)
\end{equation}
and
\begin{equation}
\Lambda_0:=\sup_{B_{g(0)}\left(x_0,\frac{2\rho}{\sqrt{K}}\right)}|Rm|(x,0)<\infty.
\end{equation}
Then for any $n\ge 3$ and $p>\frac{n+2}{2}$ there exist constants $C_0>0$ and $C>0$ such that
\begin{align}\label{rm-bd-n>2}
&|Rm|(x,t)\notag\\
\le&C_0\left\{\frac{\rho^{\frac{2n}{n+2}}e^{C(\rho+tK)}}{K^{\frac{n}{n+2}}V_{x_0}\left(\rho/\sqrt{K}\right)^{\frac{2}{n+2}}\min(t,\rho^2/K)}
\left[(K+E_p(t)^{\frac{1}{p}})t+1\right]\right\}^{\frac{n+2}{2p}}(1+\sqrt{tV_{x_0}\left(2\rho/\sqrt{K}\right)}E_p(t)^{\frac{1}{2}})^{\frac{n+4}{2p}}
\end{align}
holds for any $x\in B_{g(0)}\left(x_0,\rho/\sqrt{K}\right)$ and $0<t<T$ where 
\begin{equation*}
E_p(t)=Ce^{CKt}t\left[\Lambda_0^{2p}V_{x_0}\left(2\rho/\sqrt{K}\right)+K^{2p}(1+\rho^{-4p})V_{x_0}\left(\rho/\sqrt{K}\right)\right]
\end{equation*}
and for $n=2$ and any $p>\frac{5}{2}$ there exist constants $C_0>0$ and $C>0$ such that
\begin{align}\label{rm-bd-n=2}
&|Rm|(x,t)\notag\\
\le&C_0\left\{\frac{\rho^{\frac{4}{5}}e^{C(\rho+tK)}}{K^{\frac{2}{5}}V_{x_0}\left(\rho/\sqrt{2K}\right)^{\frac{2}{5}}\min(t,\rho^2/K)}
\left[(K+(4\rho/\sqrt{K})^{\frac{1}{p}}E_p(t)^{\frac{1}{p}})t+1\right]\right\}^{\frac{5}{2p}}\cdot\notag\\
&\qquad\cdot (1+(4\rho/\sqrt{K})\sqrt{tV_{x_0}\left(2\rho/\sqrt{K}\right)}E_p(t)^{\frac{1}{2}})^{\frac{7}{2p}}
\end{align}
holds for any $x\in B_{g(0)}\left(x_0,\rho/\sqrt{K}\right)$ and $0<t<T$.
\end{thm}

\begin{rmk}
Note that the bounds for the Riemannian curvature in \eqref{rm-bd-n>2} and \eqref{rm-bd-n=2} are slightly different from that of Theorem 1 of \cite{KMW}. When $t\to\infty$, both the right hand side of \eqref{rm-bd-n>2}, \eqref{rm-bd-n=2}, and the bound in Theorem 1 of \cite{KMW} are approximately equal to $e^{CKt}$ for some constant $C>0$.
However, for $0<t<\rho^2/K$ and $t$ close to zero, the right hand side of \eqref{rm-bd-n>2} and \eqref{rm-bd-n=2} are approximately equal to $Ct^{-\frac{n+2}{2p}}$ and $Ct^{-\frac{5}{2p}}$ respectively for some constant $C>0$, while the bound in Theorem 1 of \cite{KMW} is approximately equal to $Ct^{-\beta}$ for some constant $\beta>0$. Since the constant $\beta$ in Theorem 1 of \cite{KMW} is unknown, Theorem \ref{lp-l-infty-thm} is therefore a refinement of the result in Theorem 1 of \cite{KMW}.
\end{rmk}

\section{The main result}
\setcounter{equation}{0}
\setcounter{thm}{0}
 
We first recall a result of \cite{KMW}:

\begin{prop}\label{lp-integral-bd-prop}(Proposition 1 of \cite{KMW})
Let $g(t)$, $0\le t<T$, be a smooth solution of Ricci flow on a $n$-dimensional Riemannian manifold $M$.  Suppose there exists $x_0\in M$ and constants $K>0$, $\rho>0$, such that \eqref{ric-curv-bd-cond} holds. Then for any $n\ge 2$ and $q\ge 3$ there exists a constant $c=c(n,q)>0$ such that
\begin{equation*}
\int_{B_{g(0)}\left(x_0,\frac{\rho}{\sqrt{K}}\right)}|Rm|(x,t)^q\,dv_t\le ce^{cKt}\left\{\int_{B_{g(0)}\left(x_0,\frac{2\rho}{\sqrt{K}}\right)}|Rm|(x,0)^q\,dv_0+cK^q(1+\rho^{-2q})V_{x_0}\left(\rho/\sqrt{K},t\right)\right\}
\end{equation*}
holds for any $0\le t<T$.
\end{prop} 
\begin{proof}
A proof of this result is given in \cite{KMW}. For the sake of completeness we will give a sketch of the proof of this result in this paper. By using \eqref{rm-curv-eqn1}, the inequalities (Chapter 6 of \cite{CK} or Lemma 1 of \cite{KMW}),
\begin{equation*}
\left\{\begin{aligned}
&|\nabla Ric|^2\le\frac{1}{2}(\Delta-\1_t)|Ric|^2+CK^2|Rm|\\
&\1_tR^l_{ijk}=g^{lq}(\nabla_i\nabla_qR_{jk}+\nabla_j\nabla_iR_{kq}+\nabla_j\nabla_kR_{iq})-g^{lq}(\nabla_i\nabla_jR_{kq}+\nabla_i\nabla_kR_{jq}+\nabla_j\nabla_qR_{ik}),
\end{aligned}\right.
\end{equation*}
and a direct computation one can show that there exist constants $c_1>0$ and $c_2>0$ such that
\begin{align*}
&\frac{d}{dt}\left(\int_M|Rm|^p\phi^{2p}\,dv_t+\frac{1}{K}\int_M|Ric|^2|Rm|^{p-1}\phi^{2p}\,dv_t +c_1K\int_M|Rm|^{p-1}\phi^{2p}\,dv_t\right)\notag\\
\le &c_2K\int_M|Rm|^p\phi^{2p}\,dv_t+c_2K\int_M|Rm|^{p-1}\phi^{2p-2}\,dv_t
\end{align*} 
holds on $M\times (0,T)$ for any Lipschitz function $\phi$ with support in $B\left(x_0,\frac{2\rho}{\sqrt{K}}\right)$. Proposition \ref{lp-integral-bd-prop} then follows by choosing an appropriate cut-off function $\phi$ for the set $B\left(x_0,\frac{\rho}{\sqrt{K}}\right)$ and integrating the above differential inequality over $(0,t)$, $0<t<T$.
\end{proof}

\begin{lem}(cf. Theorem 14.3 of \cite{L})\label{sobolev-ineqn1}
Let $(M,g)$ be a complete Riemannian manifold of dimension $n\ge 3$ with Ricci curvature satisfying 
\begin{equation*}
R_{ij}\ge -(n-1)k_1\quad\mbox{ on }B(x_0,\rho)
\end{equation*} 
for some constant $k_1\ge 0$. Then there exists constants $c_1>0$ and $c_2>0$ depending only on $n$ such that for any function $f\in H_c^{1,2}(B(x_0,\rho))$ with compact support in $B(x_0,\rho)$, $f$ satisfies
\begin{equation*}
\left(\int_{B(x_0,\rho)}|f|^{\frac{2n}{n-2}}\right)^{\frac{n-2}{n}}\,dv\le c_1\frac{\rho^2e^{c_2\rho\sqrt{k_1}}}{V_{x_0}(\rho)^{2/n}}\int_{B(x_0,\rho)}|\nabla f|^2\,dv
\end{equation*}
\end{lem}

\begin{thm}\label{rm-curvature-l-infty-lp-bd-thm}
Let $g(t)$, $0\le t<T$, be a smooth solution of Ricci flow on a $n$-dimensional Riemannian manifold $M$. Suppose there exists $x_0\in M$ and constants $K>0$, $\rho>0$, such that \eqref{ric-curv-bd-cond} holds. Then for any  $n\ge 3$ and $p>\frac{n+2}{2}$ there exist constants $C_0>0$ and $C>0$ such that 
\begin{align}\label{rm-l-infty-lp-bd}
|Rm|(x,t)
\le&C_0\left\{\frac{\rho^{\frac{2n}{n+2}}e^{C(\rho+tK)}}{K^{\frac{n}{n+2}}V_{x_0}\left(\rho/\sqrt{K}\right)^{\frac{2}{n+2}}\min(t,\rho^2/K)}
\left[\left(\left(\iint_{Q_0}|Rm|^{2p}\,dv_0\,dt\right)^{\frac{1}{p}}+K\right)t+1\right]\right\}^{\frac{n+2}{2p}}\cdot\notag\\
&\qquad\cdot\left(1+\iint_{Q_0}|Rm|^p\,dv_0\,dt\right)^{\frac{n+4}{2p}}
\end{align}
holds for any $x\in B_{g(0)}\left(x_0,\rho/\sqrt{K}\right)$ and $0<t<T$ where $Q_0=B_{g(0)}\left(x_0,2\rho/\sqrt{K}\right)\times (t/4,t)$ and for $n=2$ and any $p>\frac{5}{2}$ there exist constants $C_0>0$ and $C>0$ such that 
\begin{align}\label{rm-l-infty-lp-bd2}
|Rm|(x,t)
\le&C_0\left\{\frac{\rho^{\frac{4}{5}}e^{C(\rho+tK)}}{K^{\frac{2}{5}}V_{x_0}\left(\rho/\sqrt{2K}\right)^{\frac{2}{5}}\min(t,\rho^2/K)}
\left[\left(K+\left(\frac{4\rho}{\sqrt{K}}\iint_{Q_0}|Rm|^{2p}\,dv_0\,dt\right)^{\frac{1}{p}} \right)t+1\right]\right\}^{\frac{5}{2p}}\cdot\notag\\
&\qquad\cdot\left(1+\frac{4\rho}{\sqrt{K}}\iint_{Q_0}|Rm|^p\,dv_0\,dt\right)^{\frac{7}{2p}}
\end{align}
holds for any $x\in B_{g(0)}\left(x_0,\rho/\sqrt{K}\right)$ and $0<t<T$.
\end{thm}
\begin{proof}
\noindent{\bf Case 1}: $n\ge 3$.

Let $v=|Rm|$, $0<t<T$ and $p>\frac{n+2}{2}$.
We will use a modification of the proof of Proposition 2.1 of \cite{DDD} to prove this theorem. 
By \eqref{ric-curv-bd-cond},
\begin{align}\label{vol-time-equiv}
&e^{-2Kt}g_{ij}(x,0)\le g(x,s)\le e^{2Kt}g_{ij}(x,0)\quad\forall x\in B_{g(0)}\left(x_0,\frac{2\rho}{\sqrt{K}}\right), 0\le s<T\notag\\
\Rightarrow\quad &e^{-nKt}dv_0\le dv_s\le e^{nKt}dv_0\qquad\qquad\quad\,\,\,\forall x\in B_{g(0)}\left(x_0,\frac{2\rho}{\sqrt{K}}\right), 0\le s\le t<T.
\end{align}
Let $\rho_m=\left(\rho/\sqrt{K}\right)(1+2^{-m})$  and $t_m=(1-2^{-m-1})t/2$ for any $m\ge 0$. Then $\rho_0=2\rho/\sqrt{K}$ and $t_0=t/4$. Moreover $\rho_m$ decreases to $\rho/\sqrt{K}$ and $t_m$ increases to $t/2$ as $m\to\infty$. Let $B_{\rho_m}=B_{g(0)}(x_0,\rho_m)$, $Q_m=B_{\rho_m}\times (t_m,t)$ and $Q_m^s=B_{\rho_m}\times (t_m,s)$ for any $t_m\le s\le t$. Then 
\begin{equation*}
B_{2\rho/\sqrt{K}}\times (t/4,t)=Q_0\supseteq Q_1\supseteq\cdots\supseteq Q_{m-1}\supseteq Q_m\supseteq\cdots\supseteq Q_{\infty}=B_{\rho/\sqrt{K}}\times (t/2,t).
\end{equation*}
We choose a sequence of Lipschitz continuous functions $\{\phi_m\}$ on $M\times (0,t)$ such that $0\le\phi_m\le 1$ on $M\times (0,t)$, $\phi_m(x,s)=1$ for $(x,s)\in Q_{m+1}$, $\phi_m(x,s)=0$ for $(x,s)\in M\times (0,t)\setminus Q_m$,  and satisfying 
\begin{equation}
\left\{\begin{aligned}
&|\nabla\phi_m|\le \frac{C\sqrt{K}2^m}{\rho}\quad\mbox{ in }Q_m\\
&0\le\phi_{m,t}\le \frac{C2^m}{t}\quad\mbox{ in }Q_m.
\end{aligned}\right.
\end{equation}
Let $k>0$ be a constant to be determined later and $k_m=k(1-2^{-m})$ for any $m\ge 0$. 
Multiplying \eqref{rm-curv-eqn2} by $(v-k_{m+1})_+^{p-1}\phi_m^2$ and integrating over $Q_m^s$, $t_m\le s\le t$,
\begin{align}\label{v-integral-ineqn}
&\frac{1}{p}\iint_{Q_m^s}\phi_m^2\frac{\1}{\1 t}(v-k_{m+1})_+^p\,dv_t\,dt\notag\\
&\quad +\iint_{Q_m^s}\nabla (v-k_{m+1})_+\cdot[(p-1)(v-k_{m+1})_+^{p-2}\phi_m^2
\nabla (v-k_{m+1})_++2(v-k_{m+1})_+^{p-1}\phi_m\nabla\phi_m]\,dv_t\,dt\notag\\
\le &C\iint_{Q_m^s}v^2(v-k_{m+1})_+^{p-1}\phi_m^2\,dv_t\,dt\notag\\
\le&C\left(\iint_{Q_0}v^{2p}\,dv_t\,dt\right)^{\frac{1}{p}}\left(\iint_{Q_m^s}(v-k_{m+1})_+^p\phi_m^2\,dv_t\,dt\right)^{\frac{p-1}{p}}.
\end{align}
Since $\frac{d}{dt}(dv_t)=-Rdv_t$, by \eqref{ric-curv-bd-cond},
\begin{align}\label{v-ineqn-1term}
&\iint_{Q_m^s}\phi_m^2\frac{\1}{\1 t}(v-k_{m+1})_+^p\,dv_t\,dt\notag\\
=&\int_{t_m}^s\frac{d}{dt}\left(\int_{B_{\rho_m}}(v-k_{m+1})_+^p\phi_m^2\,dv_t\right)\,dt-2\iint_{Q_m^s}(v-k_{m+1})_+^p\phi_m\phi_{m,t}\,dv_t\,dt\notag\\
&\qquad +\iint_{Q_m^s}(v-k_{m+1})_+^p\phi_m^2R\,dv_t\,dt\notag\\
\ge&\int_{B_{\rho_m}}(v(x,s)-k_{m+1})_+^p\phi_m(x,s)^2\,dv_s-C\frac{2^m}{t}\iint_{Q_m^s}(v-k_{m+1})_+^p\,dv_t\,dt\notag\\
&\qquad -CK\iint_{Q_m^s}(v-k_{m+1})_+^p\phi_m^2\,dv_t\,dt.
\end{align}
Since
\begin{equation*}
\int_{B_{\rho_m}}|\nabla(\phi_m(v-k_{m+1})_+^{\frac{p}{2}})|^2\,dv_t\le\frac{11}{10}\int_{B_{\rho_m}}\phi_m^2|\nabla (v-k_{m+1})_+^{\frac{p}{2}}|^2\,dv_t
+11\int_{B_{\rho_m}}(v-k_{m+1})_+^p|\nabla\phi_m|^2\,dv_t
\end{equation*}
and
\begin{align*}
&2\int_{B_{\rho_m}}\phi_m (v-k_{m+1})_+^{p-1}\nabla\phi_m\cdot\nabla (v-k_{m+1})_+\,dv_t\notag\\
=&\frac{4}{p}\int_{B_{\rho_m}}\phi_m (v-k_{m+1})_+^{\frac{p}{2}}\nabla\phi_m\cdot\nabla (v-k_{m+1})_+^{\frac{p}{2}}\,dv_t\\
\ge&-\frac{2}{p}\left(\frac{p-1}{p}\int_{B_{\rho_m}}\phi_m^2|\nabla (v-k_{m+1})_+^{\frac{p}{2}}|^2\,dv_t
+\frac{p}{p-1}\int_{B_{\rho_m}}(v-k_{m+1})_+^p|\nabla\phi_m|^2\,dv_t\right),
\end{align*}
\begin{align}\label{v-ineqn-2term}
&(p-1)\int_{B_{\rho_m}}\phi_m^2(v-k_{m+1})_+^{p-2}|\nabla (v-k_{m+1})_+|^2\,dv_t+2\int_{B_{\rho_m}}\phi_m (v-k_{m+1})_+^{p-1}\nabla\phi_m\cdot\nabla (v-k_{m+1})_+\,dv_t\notag\\
\ge&\frac{4(p-1)}{p^2}\int_{B_{\rho_m}}\phi_m^2|\nabla (v-k_{m+1})_+^{\frac{p}{2}}|^2\,dv_t\notag\\
&\qquad -\frac{2}{p}\left(\frac{p-1}{p}\int_{B_{\rho_m}}\phi_m^2|\nabla (v-k_{m+1})_+^{\frac{p}{2}}|^2\,dv_t
+\frac{p}{p-1}\int_{B_{\rho_m}} (v-k_{m+1})_+^p|\nabla\phi_m|^2\,dv_t\right)\notag\\
\ge&\frac{2(p-1)}{p^2}\left(\frac{10}{11}\int_{B_{\rho_m}}|\nabla((v-k_{m+1})_+^{\frac{p}{2}}\phi_m)|^2\,dv_t-10\int_{B_{\rho_m}}(v-k_{m+1})_+^p|\nabla\phi_m|^2\,dv_t\right)\notag\\
&\qquad -\frac{2}{p-1}\int_{B_{\rho_m}} (v-k_{m+1})_+^p|\nabla\phi_m|^2\,dv_t\notag\\
=&\frac{20(p-1)}{11p^2}\int_{B_{\rho_m}}|\nabla((v-k_{m+1})_+^{\frac{p}{2}}\phi_m)|^2\,dv_t-\frac{CK4^m}{\rho^2}\left(\frac{20(p-1)}{p^2}+\frac{2}{p-1}\right)\int_{B_{\rho_m}} (v-k_{m+1})_+^p\,dv_t.
\end{align}
By \eqref{v-integral-ineqn}, \eqref{v-ineqn-1term} and \eqref{v-ineqn-2term},
\begin{align}\label{v-ineqn2}
&\int_{B_{\rho_m}}v(x,s)^p\phi_m(x,s)^2\,dv_s+\iint_{Q_m^s}|\nabla((v-k_{m+1})_+^{\frac{p}{2}}\phi_m)|^2\,dv_t\,dt\notag\\
\le&C\left(\iint_{Q_0}v^{2p}\,dv_t\,dt\right)^{\frac{1}{p}}\left(\iint_{Q_m^s}(v-k_{m+1})_+^p\,dv_t\,dt\right)^{\frac{p-1}{p}}+C\left(K+\frac{2^m}{t}+\frac{K4^m}{\rho^2}\right)\iint_{Q_m^s}(v-k_{m+1})_+^p\,dv_t\,dt
\end{align}
By \eqref{vol-time-equiv} and \eqref{v-ineqn2},
\begin{align}\label{v-ineqn3}
&\sup_{t_m\le s\le t}\int_{B_{\rho_m}}v(x,s)^p\phi_m(x,s)^2\,dv_0+\iint_{Q_m}|\nabla((v-k_{m+1})_+^{\frac{p}{2}}\phi_m)|_{g(0)}^2\,dv_0\,dt\notag\\
\le&Ce^{CKt}\left\{A_1Y_m^{\frac{p-1}{p}}+\left(K+\frac{2^m}{t}+\frac{K4^m}{\rho^2}\right)Y_m\right\}
\end{align}
where
\begin{equation}\label{a1-defn}
A_1=\left(\iint_{Q_0}v^{2p}\,dv_0\,dt\right)^{\frac{1}{p}}
\end{equation}
and
\begin{equation*}
Y_m=\iint_{Q_m}(v-k_m)_+^p\,dv_0\,dt.
\end{equation*}
By Lemma \ref{sobolev-ineqn1},
\begin{align}\label{sobolev-ineqn}
\int_{B_{\rho_m}}|\nabla ((v-k_m)_+^{\frac{p}{2}}\phi_m)|_{g(0)}^2\,dv_0\ge &\frac{CV_{x_0}(\rho_m)^{\frac{2}{n}}}{\rho_m^2e^{C\rho_m\sqrt{K}}}\left(\int_{B_{\rho_m}}[(v-k_m)_+^{\frac{p}{2}}\phi_m]^{\frac{2n}{n-2}}\,dv_0\right)^{\frac{n-2}{n}}\notag\\
\ge&\frac{CKV_{x_0}(\rho/\sqrt{K})^{\frac{2}{n}}}{\rho^2e^{C\rho}}\left(\int_{B_{\rho_m}}[(v-k_m)_+^{\frac{p}{2}}\phi_m]^{\frac{2n}{n-2}}\,dv_0\right)^{\frac{n-2}{n}}.
\end{align}      
By the Holder inequality, 
\begin{align}\label{ym-ineqn10}
Y_{m+1}=&\iint_{Q_{m+1}}(v-k_{m+1})_+^p\,dv_0\,dt\notag\\
\le&\iint_{Q_m}(v-k_{m+1})_+^p\phi_m^2\,dv_0\,dt\notag\\
\le&\left(\iint_{Q_m}[(v-k_m)_+^p\phi_m^2]^{\frac{n+2}{n}}\,dv_0\,dt\right)^{\frac{n}{n+2}}|E_m|^{\frac{2}{n+2}}
\end{align}
where $E_m=\{(x,s)\in Q_m:v(x,s)>k_{m+1}\}$. By the Holder inequality and \eqref{sobolev-ineqn} (cf. proof of  proposition 3.1 of chapter 1 of \cite{D}),
\begin{align}\label{xt-sobolev-ineqn}
&\iint_{Q_m}[(v-k_m)_+^p\phi_m^2]^{\frac{n+2}{n}}\,dv_0\,dt\notag\\
=&\iint_{Q_m}[(v-k_m)_+^{\frac{p}{2}}\phi_m]^2\cdot[(v-k_m)_+^{\frac{p}{2}}\phi_m]^{\frac{4}{n}}\,dv_0\,dt\notag\\
\le&\int_{t_m}^t\left(\int_{B_{\rho_m}}[(v-k_m)_+^{\frac{p}{2}}\phi_m]^{\frac{2n}{n-2}}\,dv_0\right)^{\frac{n-2}{n}}\cdot\left(\int_{B_{\rho_m}}(v-k_m)_+^p\phi_m^2\,dv_0\right)^{\frac{2}{n}}  \,dt\notag\\
\le&\frac{C\rho^2e^{C\rho}}{KV_{x_0}(\rho/\sqrt{K})^{\frac{2}{n}}} 
\left(\iint_{Q_m}|\nabla ((v-k_m)_+^{\frac{p}{2}}\phi_m)|_{g(0)}^2\,dv_0\,dt\right)\cdot\left(
\sup_{t_n\le s\le t}\int_{B_{\rho_m}}(v-k_m)_+^p\phi_m^2\,dv_0\right)^{\frac{2}{n}}.
\end{align}
By \eqref{v-ineqn3}, \eqref{ym-ineqn10} and \eqref{xt-sobolev-ineqn},
\begin{equation}\label{ym-ineqn6}
Y_{m+1}\le\frac{C\rho^{\frac{2n}{n+2}}e^{C(\rho+tK)}}{K^{\frac{n}{n+2}}V_{x_0}(\rho/\sqrt{K})^{\frac{2}{n+2}}}\left\{A_1Y_m^{\frac{p-1}{p}}+\left(K+\frac{2^m}{t}+\frac{K4^m}{\rho^2}\right)Y_m\right\}|E_m|^{\frac{2}{n+2}}.
\end{equation}
Now (cf. proof on P.645 of \cite{DDD}), 
\begin{align}\label{measure-set-upper-bd}
Y_m=&\iint_{Q_m}(v-k_m)_+^p\,dv_0\,dt\ge(k_{m+1}-k_m)^p|E_m|=\frac{k^p}{2^{(m+1)p}}|E_m|\notag\\
\Rightarrow\quad|E_m|\le&\frac{2^{(m+1)p}}{k^p}Y_m.
\end{align}
Hemce by \eqref{ym-ineqn6} and \eqref{measure-set-upper-bd}, 
\begin{align}\label{ym-ineqn1}
Y_{m+1}\le&\frac{C\rho^{\frac{2n}{n+2}}e^{C(\rho+tK)}}{K^{\frac{n}{n+2}}V_{x_0}(\rho/\sqrt{K})^{\frac{2}{n+2}}}\left(\frac{2^{mp}}{k^p}\right)^{\frac{2}{n+2}}\left\{A_1Y_m^{1+\frac{2}{n+2}-\frac{1}{p}}+\left(K+\frac{2^m}{t}+\frac{K4^m}{\rho^2}\right)Y_m^{1+\frac{2}{n+2}}\right\}\notag\\
\le&\frac{C_1(A_1t+Kt+1)\rho^{\frac{2n}{n+2}}e^{C(\rho+tK)}}{k^\frac{2p}{n+2}K^{\frac{n}{n+2}}V_{x_0}\left(\rho/\sqrt{K}\right)^{\frac{2}{n+2}}\min(t,\rho^2/K)}
\cdot b^m\max\left(Y_m^{1+\alpha},Y_m^{1+\frac{2}{n+2}}\right)\quad\forall m\ge 0
\end{align}
for some constant $C_1>0$ where $b=4\cdot 2^{\frac{2p}{n+2}}$ and $\alpha=\frac{2}{n+2}-\frac{1}{p}$. Then $0<\alpha<\frac{2}{n+2}$. We now let $\beta>1/\alpha$ and
\begin{equation}\label{k-defn}
k=\left\{\frac{C_1(A_1t+Kt+1)\rho^{\frac{2n}{n+2}}e^{C(\rho+tK)}b^{\beta}}{K^{\frac{n}{n+2}}V_{x_0}\left(\rho/\sqrt{K}\right)^{\frac{2}{n+2}}\min(t,\rho^2/K)}
\right\}^{\frac{n+2}{2p}}\left(1+\int_{t/4}^t\int_{B_{2\rho/\sqrt{K}}}v^p\,dv_0\,dt\right)^{\frac{n+4}{2p}}.
\end{equation}
We claim that 
\begin{equation}\label{ym-ineqn2}
Y_{m+1}\le b^{-\frac{(\alpha\beta-1)[(1+\alpha)^{m+1}-(1+\alpha)]}{\alpha^2}-\frac{m}{\alpha}}<1\quad\forall m\ge 2.
\end{equation}
In order to  prove this claim we observe first that by   
\eqref{ym-ineqn1} and \eqref{k-defn},
\begin{equation}\label{y1-ineqn}
Y_1\le b^{-\beta}<1.
\end{equation}
By \eqref{ym-ineqn1}, \eqref{k-defn} and \eqref{y1-ineqn},
\begin{equation}\label{ym-ineqn3}
Y_2\le\frac{C_1(A_1t+Kt+1)\rho^{\frac{2n}{n+2}}e^{C(\rho+tK)}}{k^\frac{2p}{n+2}K^{\frac{n}{n+2}}\min(t,\rho^2/K)}
\cdot bY_1^{1+\alpha}\le b^{1-\beta-\beta (1+\alpha)}<1.
\end{equation}
Repeating the above argument we get,
\begin{equation*}
Y_{m+1}\le b^{\sum_{i=0}^mi(1+\alpha)^{m-i}-\beta\sum_{i=0}^m(1+\alpha)^i}
=b^{-\frac{(\alpha\beta-1)[(1+\alpha)^{m+1}-(1+\alpha)]}{\alpha^2}-\frac{m}{\alpha}}\quad\forall m\ge 2
\end{equation*}
and \eqref{ym-ineqn2} follows. Letting $m\to\infty$ in \eqref{ym-ineqn2}, 
\begin{equation*}
\lim_{m\to\infty}Y_{m+1}=0.
\end{equation*}
Hence $v\le k$ in $Q_{\infty}$ with $k$ given by \eqref{k-defn} and $A_1$ given by \eqref{a1-defn}. Thus \eqref{rm-l-infty-lp-bd} follows.

\noindent{\bf Case 2}: $n=2$.

Let $\4{M}=M\times\R$, $\4{g}(x,t)=g(x,t)+dx^2$, and let $\4{Rm}$, $\4{Ric}$, $\4{R}$, be the Riemannian curvature, Ricci curvature and scalar curvature of $(\4{M},\4{g}(t))$. Then
\begin{equation}\label{rm-ricci-scalar-curva-equiv}
|\4{Rm}|(x,y,t)=|Rm|(x,t), \quad |\4{Ric}|(x,y,t)=|Ric|(x,t),\quad \4{R}(x,y,t)=R(x,t)
\end{equation}
for all $x\in M$, $y\in\R$, $0\le t<T$ and
\begin{equation*}
\frac{\1 \4{g}_{ij}}{\1 t}=-2\4{R}_{ij}\quad\mbox{ in }(0,T).
\end{equation*}
Let $\4{V}_{x_0}(r)=\mbox{vol}_{\4{g}(0)}\,(B_{\4{g}(0)}((x_0,0),r))$ for any $r>0$ and $d\4{v}_0=dv_0\,dy$ be the volume element of $\4{g}(0)$. 
By case 1,
\begin{align}\label{rm-l-infty-lp-bd3}
|\4{Rm}|(x,y,t)
\le&C_0\left\{\frac{\rho^{\frac{6}{5}}e^{C(\rho+tK)}}{K^{\frac{3}{5}}\4{V}_{(x_0,0)}\left(\rho/\sqrt{K}\right)^{\frac{2}{5}}\min(t,\rho^2/K)}
\left[\left(\left(\iint_{\4{Q}_0}|Rm|^{2p}\,d\4{v}_0\,dt\right)^{\frac{1}{p}} +K\right)t+1\right]\right\}^{\frac{5}{2p}}\cdot\notag\\
&\qquad\cdot\left(1+\iint_{\4{Q}_0}|Rm|^p\,d\4{v}_0\,dt\right)^{\frac{7}{2p}}
\end{align}
holds for any $(x,y)\in B_{\4{g}(0)}\left((x_0,0),\rho/\sqrt{K}\right)$ and $0<t<T$ where $\4{Q}_0=B_{\4{g}(0)}\left((x_0,0),2\rho/\sqrt{K}\right)\times (t/4,t)$. Since $B_{\4{g}(0)}\left((x_0,0),\rho/\sqrt{K}\right)\supset B_{g(0)}\left(x_0,\rho/\sqrt{2K}\right)\times \left(-\rho/\sqrt{2K},\rho/\sqrt{2K}\right)$,
\begin{equation}\label{vol-ineqn}
\4{V}_{(x_0,0)}\left(\rho/\sqrt{K}\right)\ge\left(\sqrt{2}\rho/\sqrt{K}\right)V_{x_0}\left(\rho/\sqrt{2K}\right).
\end{equation}
Since $B_{\4{g}(0)}\left((x_0,0),2\rho/\sqrt{K}\right)\subset B_{g(0)}\left(x_0,2\rho/\sqrt{K}\right)\times \left(-2\rho/\sqrt{K},2\rho/\sqrt{K}\right)$,
by \eqref{rm-ricci-scalar-curva-equiv}, \eqref{rm-l-infty-lp-bd3} and \eqref{vol-ineqn}, we get \eqref{rm-l-infty-lp-bd2} and the theorem follows. 
\end{proof}

\begin{rmk}
By Proposition \ref{lp-integral-bd-prop}, Theorem \ref{rm-curvature-l-infty-lp-bd-thm}, Holder's inequality and \eqref{vol-time-equiv}, Theorem \ref{lp-l-infty-thm} follows.
\end{rmk}

\end{document}